\documentclass{amsart}
\usepackage{graphicx}
\usepackage{amssymb}
\usepackage{amsmath}
\usepackage{amsthm}
\usepackage{subfigure}
\usepackage{enumerate}
\usepackage{bbm}
\usepackage{calligra}
\DeclareMathAlphabet{\mathcalligra}{T1}{calligra}{m}{n}

\newtheorem{thm}{Theorem}[section]
\newtheorem{cor}[thm]{Corollary}

\newtheorem{lem}[thm]{Lemma}
\newtheorem{prop}[thm]{Proposition}
\theoremstyle{definition}
\newtheorem{defn}[thm]{Definition}
\newtheorem{rem}[thm]{Remark}

\newtheorem*{defn*}{Definition}
\newtheorem*{rems*}{Remarks}
\newtheorem*{rem*}{Remark}

\numberwithin{equation}{section}

\newcommand{\Eq}{{\text{E}}}

\newcommand{\M}{{M}}

\begin{document}

\title[The improved isoperimetric inequality and the Wigner caustic] {The improved isoperimetric inequality\linebreak and the Wigner caustic of planar ovals}
\author{ Micha\l{} Zwierzy\'nski}
\address{Warsaw University of Technology\\
Faculty of Mathematics and Information Science\\
Plac Politechniki 1\\
00-661 Warsaw\\
Poland\\}

\email{zwierzynskim@mini.pw.edu.pl}
\thanks{The work of M. Zwierzy\'nski was partially supported by NCN grant no. DEC-2013/11/B/ST1/03080. }



\begin{abstract}
The classical isoperimetric inequality in the Euclidean plane $\mathbb{R}^2$ states that for a simple closed curve $\M$ of the length $L_{\M}$, enclosing a region of the area $A_{\M}$, one gets
\begin{align*}
L_{\M}^2\geqslant 4\pi A_{\M}.
\end{align*}
In this paper we present the \textit{improved isoperimetric inequality}, which states that if $\M$ is a closed regular simple convex curve, then
\begin{align*}
L_{\M}^2\geqslant 4\pi A_{\M}+8\pi\left|\widetilde{A}_{E_{\frac{1}{2}}(\M)}\right|,
\end{align*}
where $\widetilde{A}_{E_{\frac{1}{2}}(\M)}$ is an oriented area of the Wigner caustic of $\M$, and the equality holds if and only if $\M$ is a curve of constant width.

Furthermore we also present a stability property of the improved isoperimetric inequality (near equality implies curve nearly of constant width).

The Wigner caustic is an example of an affine $\lambda$-equidistant (for $\displaystyle\lambda=\frac{1}{2}$) and the improved isoperimetric inequality is a consequence of certain bounds of oriented areas of affine equidistants.
\end{abstract}

\maketitle

\section{Introduction}
The \textit{classical isoperimetric inequality} in the Euclidean plane $\mathbb{R}^2$ states that:

\begin{thm}(Isoperimetric inequality)
Let $\M$ be a simple closed curve of the length $L_{\M}$, enclosing a region of the area $A_{\M}$, then
\begin{align}\label{IsoperimetricIneq}
L_{\M}^2\geqslant 4\pi A_{\M},
\end{align}

and the equality (\ref{IsoperimetricIneq}) holds if and only if $\M$ is a circle.

\end{thm}

This fact was already known in ancient Greece. The first mathematical proof was given in the nineteenth century by Steiner \cite{S1}. After that, there have been many new proofs, generalizations, and applications of this famous theorem, see for instance \cite{C1, G1, G4, H3,  L1, PX1, R1, S1}, and the literature therein. In 1902 Hurtwiz \cite{H3} and later Gao \cite{G1} showed the \textit{reverse isoperimetric inequality}.

\begin{thm}(Reverse isoperimetric inequality)
Let $K$ be a stricly convex domain whose support function $p$ has the property that $p''$ exists and is absolutely continuous, and let $\widetilde{A}$ denote the oriented area of the evolute of the boundary curve of $K$. Let $L_{K}$ be the perimeter of $K$ and $A_{K}$ be the area of $K$. Then
\begin{align}\label{ReverseIsoperimetricIneq}
L_{\M}^2\leqslant 4\pi A_{\M}+\pi|\widetilde{A}|,
\end{align}
Equality holds if and only if $p(\theta)=a_0+a_1\cos\theta+b_1\sin\theta+a_2\cos 2\theta+b_2\sin 2\theta$.
\end{thm}

In this paper we present bounds of oriented areas of affine equidistants and thanks to it we will prove the \textit{improved isoperimetric inequality}, which states that if $\M$ is a closed regular simple convex curve, then
\begin{align*}
L_{\M}^2\geqslant 4\pi A_{\M}+8\pi\left|\widetilde{A}_{E_{\frac{1}{2}}(\M)}\right|,
\end{align*}
where $\widetilde{A}_{E_{\frac{1}{2}}(\M)}$ is an oriented area of the Wigner caustic of $\M$, and the equality holds if and only if $\M$ is a curve of constant width. This is very interesting that the absolute value of the oriented area of the Wigner caustic improves the classical isoperimetric inequality and also gives the exact link between the area and the length of constant width curves.

The family of affine $\lambda$ - equidistants arises as the counterpart of parallels or offsets in Euclidean geometry. An affine equidistant for us is the set of points of chords connecting points on $\M$ where tangent lines to $\M$ are parallel, which divide the chord segments between the base points with a fixed ratio $\lambda$, also called the \textit{affine time}. When in affine $\lambda$-equidistants the ratio $\lambda$ is equal to $\displaystyle\frac{1}{2}$ then this set is also known as the \textit{Wigner caustic}. The Wigner caustic of a smooth convex closed curve on affine symplectic plane was first introduced by Berry, in his celebrated 1977 paper \cite{B1} on the semiclassical limit of Wigner's phase-space representation of quantum states. They are many papers considering affine equidistants, see for instance \cite{C2, CDR1, DMR1, DR1, DRS1,  DZ1, G3, GWZ1, RZ1, Z1}, and the literature therein. The Wigner caustic is also known as the \textit{area evolute}, see \cite{C2, G3}.


\section{Geometric quantities, affine equidistants and Fourier series}

Let $\M$ be a smooth planar curve, i.e. the image of the $C^{\infty}$ smooth map from an interval to $\mathbb{R}^2$. A smooth curve is \textit{closed} if it is the image of a $C^{\infty}$ smooth map from $S^1$ to $\mathbb{R}^2$. A smooth curve is \textit{regular} if its velocity does not vanish. A regular closed curve is \textit{convex} if its signed curvature has a constant sign. An \textit{oval} is a smooth closed convex curve which is simple, i.e. it has no selfintersections. In our case it is enough to consider $C^2$ - smooth curves.

\begin{defn}\label{parallelpair}
A pair $a,b\in\M$ ($a\neq b$) is called the \textit{parallel pair} if tangent lines to $\M$ at points $a,b$ are parallel.
\end{defn}

\begin{defn}\label{chord}
A \textit{chord} passing through a pair $a,b\in\M$ is the line:
$$l(a,b)=\left\{\lambda a+(1-\lambda)b\ \big| \lambda\in\mathbb{R}\right\}.$$
\end{defn}

\begin{defn}\label{equidistantSet}
An affine $\lambda$-equidistant is the following set.
$$\Eq_{\lambda}(\M)=\left\{\lambda a+(1-\lambda)b\ \big|\ a,b \text{ is a parallel pair of } \M\right\}.$$

The set $\Eq_{\frac{1}{2}}(\M)$ will be called the \textit{Wigner caustic} of $\M$.
\end{defn}

Note that, for any given $\lambda\in\mathbb{R}$ we have an equality $\Eq_{\lambda}(\M)=\Eq_{1-\lambda}(\M)$. Thus, the case $\displaystyle\lambda=\frac{1}{2}$ is special. In particular we have also equalities $\Eq_0(\M)=\Eq_1(\M)=\M$.

It is well known that if $\M$ is a generic oval, then $\Eq_{\lambda}(\M)$ are smooth closed curves with cusp singularities only \cite{B1, GZ1}, the number of cusps of $\Eq_{\frac{1}{2}}(\M)$ is odd and not smaller than $3$ \cite{B1, G3} and the number of cusps of $\Eq_{\lambda}(\M)$ for a generic value of $\displaystyle\lambda\neq\frac{1}{2}$ is even \cite{DZ1}.

\begin{defn}
An oval is said to have \textit{constant width} if the distance between every pair of parallel tangent lines is constant. This constant is called the \textit{width} of the curve.
\end{defn}

\begin{figure}[h]
\centering
\includegraphics[scale=0.35]{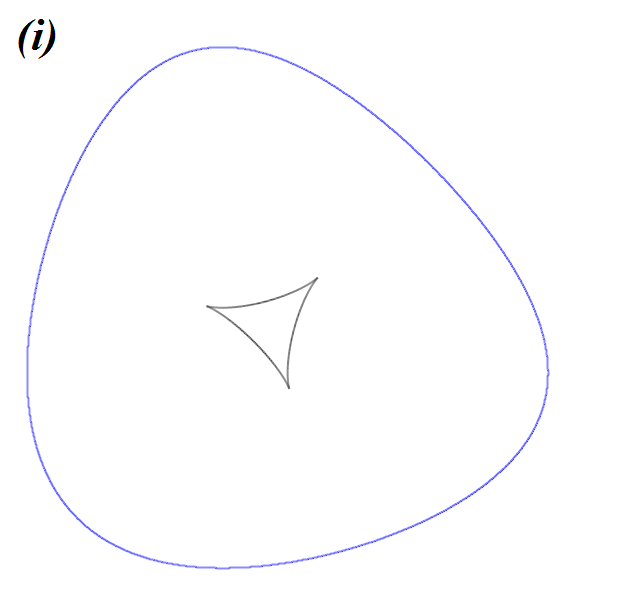}
\includegraphics[scale=0.35]{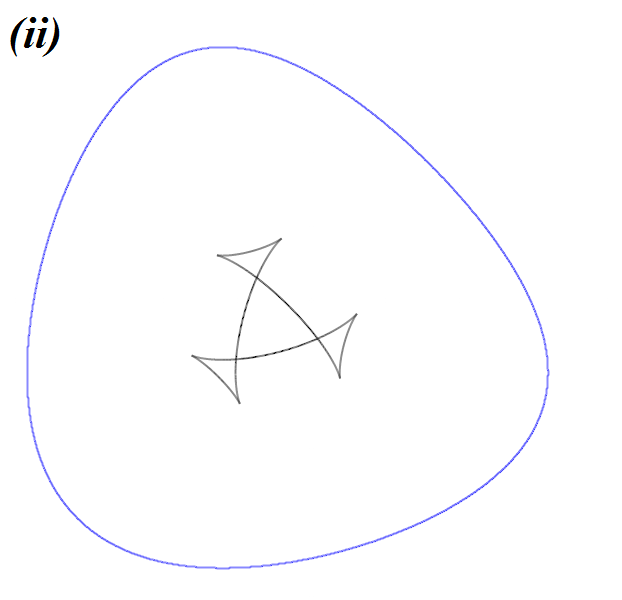}
\caption{An oval $\M$ and (i) $\Eq_{\frac{1}{2}}(\M)$, (ii) $\Eq_{\frac{2}{5}}(\M)$.}
\label{Picture1}
\end{figure}

Let us recall some basic facts about plane ovals which will be used later. The details can be found in the classical literature \cite{G4, H2}.

Let $\M$ be a positively oriented oval. Take a point $O$ inside $\M$ as the origin of our frame. Let $p$ be the oriented perpendicular distance from $O$ to the tangent line at a point on $\M$, and $\theta$ the oriented angle from the positive $x_1$-axis to this perpendicular ray. Clearly, $p$ is a single-valued periodic function of $\theta$ with period $2\pi$ and the parameterization of $\M$ in terms of $\theta$ and $p(\theta)$ is as follows
\begin{align}\label{ParameterizationM}
\gamma(\theta)=\big(\gamma_1(\theta),\gamma_2(\theta)\big)=\big(p(\theta)\cos\theta-p'(\theta)\sin\theta, p(\theta)\sin\theta+p'(\theta)\cos\theta\big).
\end{align}
The couple $\big(\theta, p(\theta)\big)$ is usually called the \textit{polar tangential coordinate} on $\M$, and $p(\theta)$ its \textit{Minkowski's support function}.

Then, the curvature $\kappa$ of $\M$ is in the following form
\begin{align}\label{CurvatureM}
\displaystyle \kappa(\theta)=\frac{d\theta}{ds}=\frac{1}{p(\theta)+p''(\theta)}>0,
\end{align}
or equivalently, the radius of a curvature $\rho$ of $\M$ is given by
\begin{align*}
\rho(\theta)=\frac{ds}{d\theta}=p(\theta)+p''(\theta).
\end{align*}
Let $L_{\M}$ and $A_{\M}$ be the length of $\M$ and the area it bounds, respectively. Then one can get that
\begin{align}\label{CauchyFormula}
L_{\M}=\int_{\M}ds=\int_0^{2\pi}\rho(\theta)d\theta=\int_0^{2\pi}p(\theta)d\theta,
\end{align}
and
\begin{align}\label{BlaschkeFormula}
A_{\M} & =\frac{1}{2}\int_{\M}p(\theta)ds\\
\nonumber	&=\frac{1}{2}\int_0^{2\pi}p(\theta)\left[p(\theta)+p''(\theta)\right]d\theta=\frac{1}{2}\int_0^{2\pi}\left[p^2(\theta)-p'^2(\theta)\right]d\theta.
\end{align}
(\ref{CauchyFormula}) and (\ref{BlaschkeFormula}) are known as \textit{Cauchy's formula} and \textit{Blaschke's formula}, respectively.

Since the Minkowski support function of $\M$ is smooth bounded and $2\pi$--periodic, its Fourier series is in the form
\begin{align}\label{Fourierofp}
p(\theta)=a_0+\sum_{n=1}^{\infty}\big(a_n\cos n\theta+b_n\sin n\theta\big).
\end{align}
Differentiation of (\ref{Fourierofp}) with respect to $\theta$ gives
\begin{align}\label{Fourierofpprime}
p'(\theta)=\sum_{n=1}^{\infty}n\big(-a_n\sin n\theta+b_n\cos n\theta\big).
\end{align}
By (\ref{Fourierofp}), (\ref{Fourierofpprime}) and the Parseval equality one can express $L_{\M}$ and $A_{\M}$ in terms of Fourier coefficients of $p(\theta)$ in the following way.
\begin{align}
\label{Lengthofmfourier} L_{\M} &=2\pi a_0.\\
\label{Areaofmfourier} A_{\M} &=\pi a_0^2-\frac{\pi}{2}\sum_{n=2}^{\infty}(n^2-1)(a_n^2+b_n^2).
\end{align}

One can notice that $\gamma(\theta),\gamma(\theta+\pi)$ is a parallel pair of $\M$, hence $\gamma_{\lambda}$ - the parameterization of $\Eq_{\lambda}(\M)$ is as follows
\begin{align}\label{ParameterizationEqM}
\gamma_{\lambda}(\theta) 
	&=\big(\gamma_{\lambda, 1}(\theta), \gamma_{\lambda, 2}(\theta)\big)\\
\nonumber	&=\lambda\gamma(\theta)+(1-\lambda)\gamma(\theta+\pi) \\
\nonumber	&=\left(P_{\lambda}(\theta)\cos\theta-P'_{\lambda}\sin\theta, P_{\lambda}(\theta)\sin\theta+P'_{\lambda}(\theta)\cos\theta\right),
\end{align}
where $P_{\lambda}(\theta)=\lambda p(\theta)-(1-\lambda)p(\theta+\pi)$, $\theta\in[0,2\pi]$. Furthermore if $\displaystyle\lambda=\frac{1}{2}$, then the map $\M\ni\gamma(\theta)\mapsto\gamma_{\frac{1}{2}}(\theta)\in\Eq_{\frac{1}{2}}(\M)$ for $\theta\in[0,2\pi]$ is the double covering of the Wigner caustic of $\M$.


\section{Oriented areas of equidistants and the improved isoperimetric inequality}

Let $L_{\Eq_{\lambda}(\M)}, \widetilde{A}_{E_{\lambda}(\M)}$ denote the length of $\Eq_{\lambda}(\M)$ and the oriented area of $\Eq_{\lambda}(\M)$, respectively.

Similarly like in \cite{DZ1} we can show the following proposition.
\begin{prop}\cite{DZ1}
Let $M$ be an oval. Then
\begin{enumerate}[(i)]
\item if $\displaystyle\lambda\neq\frac{1}{2}$, then $L_{\Eq_{\lambda}(\M)}\leqslant\big(|\lambda|+|1-\lambda|\big)L_{\M}$.

In particular if $\displaystyle\lambda\in\left(0,\frac{1}{2}\right)\cup\left(\frac{1}{2},1\right)$, then $L_{\Eq_{\lambda}(\M)}\leqslant L_{\M}$.
\item $2L_{\Eq_{\frac{1}{2}}(\M)}\leqslant L_{\M}$.
\end{enumerate}
\end{prop}
\begin{proof}
The parameterization of $\M$ and $\Eq_{\lambda}(\M)$ is like in (\ref{ParameterizationM}) and (\ref{ParameterizationEqM}), respectively. Then
\begin{align*}
L_{\Eq_{\lambda}(\M)}
	&=\int_0^{2\pi}|\gamma'_{\lambda}(\theta)|d\theta\\
	&=\int_0^{2\pi}\big|\lambda\gamma'(\theta)+(1-\lambda)\gamma'(\theta+\pi)\big|d\theta\\
	&\leqslant |\lambda|\int_0^{2\pi}|\gamma'(\theta)|d\theta+|1-\lambda|\int_0^{2\pi}|\gamma'(\theta+\pi)|d\theta\\
	&=\big(|\lambda|+|1-\lambda|\big)L_{\M}.
\end{align*}
If $\displaystyle\lambda=\frac{1}{2}$, then the map $\M\ni\gamma(\theta)\mapsto\gamma_{\frac{1}{2}}(\theta)\in\Eq_{\frac{1}{2}}(\M)$ for $\theta\in[0,2\pi]$ is the double covering of the Wigner caustic of $\M$. Thus $2L_{\Eq_{\frac{1}{2}}(\M)}\leqslant L_{\M}$.
\end{proof}

\begin{thm}\label{ThmBoundedAreas}
Let $\M$ be a positively oriented oval of the length $L_{\M}$, enclosing the region of the area $A_{\M}$. Let $\widetilde{A}_{E_{\lambda}(\M)}$ denote an oriented area of $E_{\lambda}(\M)$. Then
\begin{enumerate}[(i)]
\item if $\displaystyle\lambda\in\left(0,\frac{1}{2}\right)\cup\left(\frac{1}{2},1\right)$, then
\begin{align}\label{ThmA}
A_{\M}-\frac{\lambda(1-\lambda)}{\pi}L_{\M}^2\leqslant\widetilde{A}_{E_{\lambda}(\M)}\leqslant (2\lambda-1)^2A_{\M}.
\end{align}
\item if $\displaystyle\lambda=\frac{1}{2}$, then
\begin{align}\label{ThmB}
A_{\M}-\frac{L_{\M}^2}{4\pi}\leqslant 2\widetilde{A}_{E_{\frac{1}{2}}(\M)}\leqslant 0.
\end{align}
\item if $\lambda\in(-\infty,0)\cup(1,\infty)$, then
\begin{align}\label{ThmC}
(2\lambda-1)^2A_{\M}\leqslant\widetilde{A}_{E_{\lambda}(\M)}\leqslant A_M-\frac{\lambda(1-\lambda)}{\pi}L_{\M}^2.
\end{align}
\item for all $\displaystyle\lambda\neq 0, \frac{1}{2}, 1$ $\M$ is a curve of constant width if and only if
\begin{align}\label{ThmD}
\widetilde{A}_{E_{\lambda}(\M)}=A_{\M}-\frac{\lambda(1-\lambda)}{\pi}L_{\M}^2.
\end{align}
\item $\M$ is a curve of constant width if and only if
\begin{align}\label{ThmE}
2\widetilde{A}_{E_{\frac{1}{2}}(\M)}=A_{\M}-\frac{L_{\M}^2}{4\pi}.
\end{align}
\end{enumerate}
\end{thm}
\begin{proof}
Let (\ref{ParameterizationEqM}) be the parameterization of $\Eq_{\lambda}(\M)$. Then the oriented area of $\Eq_{\lambda}(\M)$ is equal to
\begin{align*}
\widetilde{A}_{E_{\lambda}(\M)}
	&=\frac{1}{2}\int_{\Eq_{\lambda}(\M)}\gamma_{\lambda, 1}d\gamma_{\lambda, 2}-\gamma_{\lambda, 2}d\gamma_{\lambda, 1}& \\
	&=\frac{1}{2}\int_{0}^{2\pi}\Big[\left(P_{\lambda}(\theta)\cos\theta-P'_{\lambda}(\theta)\sin\theta\right)\left(P_{\lambda}(\theta)+P''_{\lambda}(\theta)\right)\cos\theta \\
	&\qquad\qquad\quad +\left(P_{\lambda}(\theta)\sin\theta+P'_{\lambda}\cos\theta\right)\left(P_{\lambda}(\theta)+P''_{\lambda}(\theta)\right)\sin\theta\Big]d\theta\\
	&=\frac{1}{2}\int_0^{2\pi}\Big[P_{\lambda}^2(\theta)+P_{\lambda}(\theta)P''_{\lambda}(\theta)\Big]d\theta\\
	&=\frac{1}{2}\int_0^{2\pi}\Big[P_{\lambda}^2(\theta)-P'^2_{\lambda}(\theta)\Big]d\theta\\
	&=\frac{1}{2}\int_0^{2\pi}\Big[\big(\lambda p(\theta)-(1-\lambda)p(\theta+\pi)\big)^2-\big(\lambda p'(\theta)-(1-\lambda)p'(\theta+\pi)\big)^2\Big]d\theta\\
	&=\lambda^2\cdot\frac{1}{2}\int_{0}^{2\pi}\left[p^2(\theta)-p'^2(\theta)\right]d\theta+(1-\lambda)^2\cdot\frac{1}{2}\int_0^{2\pi}\left[p^2(\theta+\pi)-p'^2(\theta+\pi)\right]d\theta\\
	&\qquad -2\lambda(1-\lambda)\cdot\frac{1}{2}\int_0^{2\pi}\left[p(\theta)p(\theta+\pi)-p'(\theta)p'(\theta+\pi)\right]d\theta\\
	&=\left(2\lambda^2-2\lambda+1\right)A_{\M}-2\lambda(1-\lambda)\cdot\frac{1}{2}\int_0^{2\pi}\left[p(\theta)p(\theta+\pi)-p'(\theta)p'(\theta+\pi)\right]d\theta.
\end{align*}
Let $\displaystyle \Psi_{\M}=\frac{1}{2}\int_0^{2\pi}\left[p(\theta)p(\theta+\pi)-p'(\theta)p'(\theta+\pi)\right]d\theta$, then the oriented area of $\Eq_{\lambda}(\M)$ is in the following form.
\begin{align}
\label{OrientedAreaFormula}\widetilde{A}_{E_{\lambda}(\M)}=(2\lambda^2-2\lambda+1)A_{\M}-2\lambda(1-\lambda)\Psi_{\M}.
\end{align}

Let us find formula for $\Psi_{\M}$ in terms of coefficients of Fourier series of the Minkowski support function $p(\theta)$. By (\ref{Fourierofp}) and (\ref{Fourierofpprime}) there are also equalities:
\begin{align}
\nonumber p(\theta+\pi) &=a_0+\sum_{n=1}^{\infty}(-1)^n\big(a_n\cos(n\theta)+b_n\sin(n\theta)\big), \\ 
\nonumber p'(\theta+\pi) &=\sum_{n=1}^{\infty}(-1)^nn\big(-a_n\sin(n\theta)+b_n\cos(n\theta)\big),\\
\label{Phithetadef}\Psi_{\M} &=\pi a_0^2-\frac{\pi}{2}\sum_{n=2}^{\infty}(-1)^n(n^2-1)(a_n^2+b_n^2).
\end{align}
Then bounds of $\Psi_{\M}$ are as follows.
\begin{align}\label{IneqPhitheta}
\pi a_0^2-\frac{\pi}{2}\sum_{n=2}^{\infty}(n^2-1)(a_n^2+b_n^2)\leqslant \Psi_{\M} &\leqslant \pi a_0^2+\frac{\pi}{2}\sum_{n=2}^{\infty}(n^2-1)(a_n^2+b_n^2).
\end{align}
By (\ref{Lengthofmfourier}) and (\ref{Areaofmfourier}) one can rewrite (\ref{IneqPhitheta}) as
\begin{align}\label{InePhithetaBetter}
A_{\M}\leqslant \Psi_{\M} &\leqslant \frac{L_{\M}^2}{2\pi}-A_{\M}
\end{align}
Then using  (\ref{InePhithetaBetter}) to (\ref{OrientedAreaFormula}) and because the map $\M\ni\gamma(\theta)\mapsto\gamma_{\frac{1}{2}}(\theta)\in\Eq_{\frac{1}{2}}(\M)$ for $\theta\in[0,2\pi]$ is the double covering of the Wigner caustic of $\M$, one can obtain (\ref{ThmA}), (\ref{ThmB}), (\ref{ThmC}).

To prove (\ref{ThmD}) and (\ref{ThmE}) let us notice that $\M$ is a curve of constant width if and only if coefficients $a_{2n}, b_{2n}$ for $n\geqslant 1$ in the Fourier series of $p(\theta)$ are all equal to zero \cite{F1, G4}, and then the formula (\ref{Phithetadef}) of $\Psi_{\M}$ is as follows:
\begin{align*}
\Psi_{\M} &=\pi a_0^2-\frac{\pi}{2}\sum_{n=2}^{\infty}(-1)^n(n^2-1)(a_n^2+b_n^2)\\
		&=\pi a_0^2+\frac{\pi}{2}\sum_{n=2, n\text{ is odd}}^{\infty}(n^2-1)(a_n^2+b_n^2)\\
		&=2\pi a_0^2-A_{\M}\\
		&=\frac{L_{\M}^2}{2\pi}-A_{\M}.
\end{align*}
\end{proof}

A simple consequence of the above theorem is the following remark and also the main result, the improved isoperimetric inequality for planar ovals.

\begin{rem}
Let $\M$ be an positively oriented oval. Then Theorem \ref{ThmBoundedAreas} gives us that $\widetilde{A}_{E_{\frac{1}{2}}(\M)}\leqslant 0$, which leads to the fact that the Wigner caustic of $\M$ has a reversed orientation against that of the original curve $\M$.
\end{rem}

\begin{thm}(The improved isoperimetric inequality)\label{ImprovedIsoperimetricInequalityThm}
If $\M$ is an oval of the length $L_{\M}$, enclosing a region of the area $A_{\M}$, then
\begin{align}\label{ImprovedIsoperimetricInequality}
L_{\M}^2\geqslant 4\pi A_{\M}+8\pi|\widetilde{A}_{E_{\frac{1}{2}}(\M)}\big|,
\end{align}
where $\widetilde{A}_{E_{\frac{1}{2}}(\M)}$ is an oriented area of the Wigner caustic of $\M$, and equality (\ref{ImprovedIsoperimetricInequality}) holds if and only if $\M$ is a curve of constant width.
\end{thm}

\begin{rem}
The improved isoperimetric inequality becomes isoperimetric inequality if and only if $\widetilde{A}_{E_{\frac{1}{2}}(\M)}=0$, that means when $\M$ has the center of symmetry.
\end{rem}

\begin{figure}[h]
\centering
\includegraphics[scale=0.35]{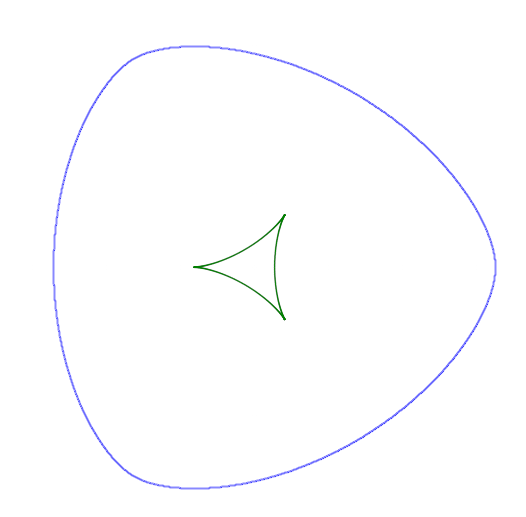}
\caption{An oval $\M_3$ of constant width and the Wigner caustic of $\M_3$. The Minkowski support function of $\M_3$ is $p_3(\theta)=\cos 3\theta+11$.}
\label{PictureConstantWidth}
\end{figure}

\begin{thm}(Barbier's Theorem)\cite{B2}
Let $\M$ be a curve of constant width $w$. Then the length of $\M$ is equal to $\pi w$.
\end{thm}

By Barbier's Theorem and Theorem \ref{ImprovedIsoperimetricInequalityThm} one can get the following corollary.
\begin{cor}
Let $\M$ be an oval of constant width $w$, enclosing region with the area $A_{\M}$. Then
\begin{align*}
A_{\M}=\frac{\pi w^2}{4}-2\left|\widetilde{A}_{E_{\frac{1}{2}}(\M)}\right|,
\end{align*}
where $\widetilde{A}_{E_{\frac{1}{2}}(\M)}$ is an oriented area of the Wigner caustic of $\M$.
\end{cor}

One can check that curve $M$ for which $p(\theta)=\cos 3\theta+11$ (see Fig. \ref{PictureConstantWidth}) is an oval of constant width $w=22$,  enclosing the area $A_{\M}=117\pi$, and a signed area of the Wigner caustic of $\M$ is equal to $\widetilde{A}_{E_{\frac{1}{2}}(\M)}=-2\pi$. 

In Proposition \ref{PropCusps} we present explicitly curves for which the Wigner caustic has ecactly $2n+1$ cusps (this number must be odd and not smaller than $3$ - \cite{B1, G3}). $M_7$ and $\Eq_{\frac{1}{2}}(\M_7)$ are shown in Fig. \ref{PictureConstantWidth7cusps}.
 
\begin{prop}\label{PropCusps}
Let $n$ be a positive integer. Let $M_{2n+1}$ be a curve for which $p_{2n+1}(\theta)=\cos\big[\theta(2n+1)\big]+(2n+1)^2+2$ is its Minkowski support function. Then $\M$ is an oval of constant width and $\Eq_{\frac{1}{2}}(\M_{2n+1})$ has exactly $2n+1$ cusps.
\end{prop}
\begin{proof}
$\M$ is singular if and only if $p_{2n+1}(\theta)+p''_{2n+1}(\theta)=0$, but one can check that is impossible. $\M$ is curve of constant width, because $p_{2n+1}(\theta)+p_{2n+1}(\theta+\pi)$ is constant for all values of $\theta$.

The cusp of the Wigner caustic of an oval appears when curvatures of the original curve at points in parallel pair are equal \cite{B1}.

One can check that equation $\kappa(\theta)=\kappa(\theta+\pi)$, where $\theta\in\left[0,\pi\right]$ holds if and only if $\displaystyle\theta=\frac{\pi+2k\pi}{4n+2}$ for $k\in\{0,1,2,\ldots,2n\}$.
\end{proof}

\begin{figure}[h]
\centering
\includegraphics[scale=0.35]{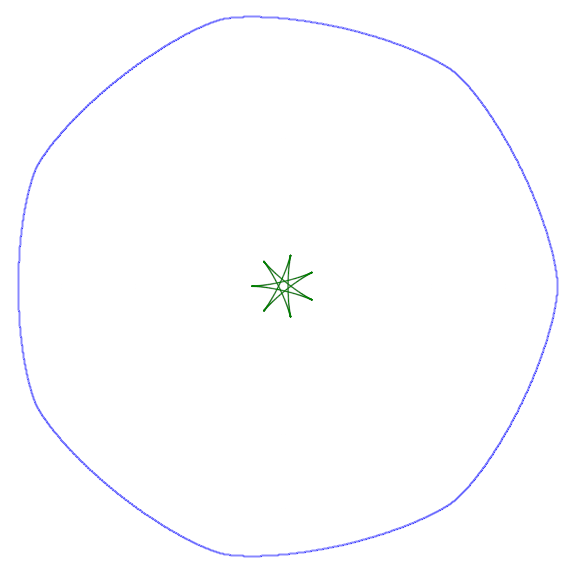}
\includegraphics[scale=0.35]{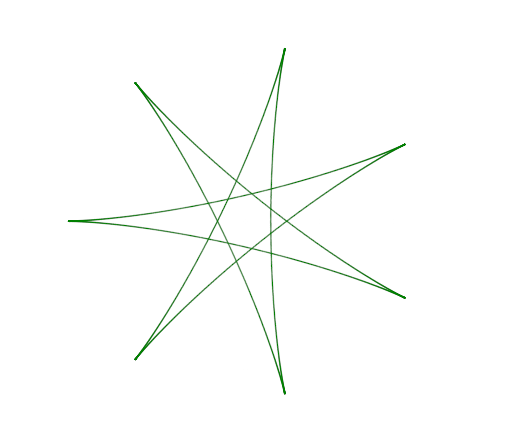}
\caption{
An oval $\M_7$ of constant width and $\Eq_{\frac{1}{2}}(\M_7)$. The Minkowski support function of $\M_7$ is $p_7(\theta)=\cos 7\theta+51$.}
\label{PictureConstantWidth7cusps}
\end{figure}


\section{The stability of the improved isoperimetric inequality}

A bounded convex subset in $\mathbb{R}^n$ is said to be an \textit{$n$ -- dimensional convex body} if it is closed and has interior points. Let $\mathcal{C}^n$ denote the set of all $n$ -- dimensional convex bodies.

There are many important inequalities in convex geometry and differential geometry, such as the isoperimetric inequality, the Brunn--Minkowski inequality,\linebreak Aleksandrov--Fenchel inequality. The stability property of them are of great interest in geometric analysis, see \cite{F2, G2, H3, PX1, S2} and the literature therein. An inequality in convex geometry can be written as
\begin{align}\label{IneqConvexGeometry} 
\Phi(K)\geqslant 0,
\end{align}
where $\Phi:\mathcal{C}^n\to\mathbb{R}$ is a real-valued function and (\ref{IneqConvexGeometry}) holds for all $K\in\mathcal{C}^n$. Let $\mathcal{C}^n_{\Phi}$ be a subset of $\mathcal{C}^n$ for which equality in (\ref{IneqConvexGeometry}) holds. 

For example, if $n=2$, and like in previous sections let $L_{\partial K}$ denote the length of the boundary of $K$, let $A_{\partial K}$ denote the area enclosed by $\partial K$ (i.e. the area of $K$), $\Phi(K)=L_{\partial K}^2-4\pi A_{\partial K}$. Then inequality $\Phi(K)\geqslant 0$ is then the classical isoperimetric inequality in $\mathbb{R}^2$. In this case $\mathcal{C}^2_{\Phi}$ is a set of disks.

In this section we will study stability properties associated with (\ref{IneqConvexGeometry}). We ask if $K$ must by close to a member of $\mathcal{C}^n_{\Phi}$ whenever $\Phi(K)$ is close to zero. \linebreak Let $d:\mathcal{C}^n\times\mathcal{C}^n\to\mathbb{R}$ denoted in some sense the deviation between two convex bodies. $d$ should satisfy two following conditions:
\begin{enumerate}[(i)]
\item $d(K,L)\geqslant 0$ for all $K,L\in\mathcal{C}^n$,
\item $d(K,L)=0$ if and only if $K=L$.
\end{enumerate}

If $\Phi, \mathcal{C}^n_{\Phi}$ and $d$ are given, then the \textit{stability problem} associated with (\ref{IneqConvexGeometry}) is as follows.

\textit{Find positive constants $c,\alpha$ such that for each $K\in\mathcal{C}^n$, there exists $N\in C^n_{\Phi}$ such that 
\begin{align}\label{StabilityIneq}
\Phi(K)\geqslant cg^{\alpha}(K,N).
\end{align}}

From this point let us assume that $n=2$ and by Theorem \ref{ImprovedIsoperimetricInequalityThm} let
\begin{align}\label{IneqConvexProblem}
\Phi(K)=L_{\partial K}^2-4\pi A_{\partial K}-8\pi\left|\widetilde{A}_{E_{\frac{1}{2}}(\M)}\right|\geqslant 0.
\end{align}
From Theorem \ref{ImprovedIsoperimetricInequalityThm} one can see that $C^2_{\Phi}$ consists of bodies of constant width.

Let us recall two $d$ measure functions.

Let $K$ and $N$ be two convex bodies with respective support functions $p_{\partial K}$ and $p_{\partial N}$. Usually to measure the deviation between $K$ and $N$ one can use the \textit{Hausdorff distance},
\begin{align}\label{HausdorffDistance}
d_{\infty}(K,N)=\max_{\theta}\Big|p_{\partial K}(\theta)-p_{\partial N}(\theta)\Big|.
\end{align}

Another such measure is the measure that corresponds to the $L_2$-metric in the function space. It is defined by
\begin{align}\label{LTwoDistance}
d_2(K,N)=\left(\int_0^{2\pi}\Big|p_{\partial K}(\theta)-p_{\partial N}(\theta)\Big|^2d\theta\right)^{\frac{1}{2}}.
\end{align}
It is obvious that $d_{\infty}(K,N)=0$ (or $d_2(K,N)=0$) if and only of $K=N$.

\begin{lem}\label{LemmaTrig}
Let $c_k, d_k\in\mathbb{R}$ for $k\in\{1,2,\ldots,n\}$. Then
\begin{align*}
\max_{\theta}\left|\sum_{k=1, \text{k is odd}}^n\big(c_k\cos k\theta+d_k\sin k\theta\big)\right|\leqslant\max_{\theta}\left|\sum_{k=1}^n\big(c_k\cos k\theta+d_k\sin k\theta\big)\right|.
\end{align*}
\end{lem}
\begin{proof}
Let 
\begin{align*}
f_{odd}(\theta) &=\sum_{k=1, \text{k is odd}}^n\big(c_k\cos k\theta+d_k\sin k\theta\big),\\
f_{even}(\theta) &=\sum_{k=1, \text{k is even}}^n\big(c_k\cos k\theta+d_k\sin k\theta\big),\\
f(\theta) & =f_{odd}(\theta)+f_{even}(\theta).
\end{align*}
One can see that $f_{odd}$ is bounded, $2\pi$-periodic and $f_{odd}(\theta)=-f_{odd}(\theta+\pi)$. 
Let $\theta_0$ be an argument for which $\displaystyle f_{odd}(\theta_0)=\max_{\theta}f_{odd}(\theta)$, then $\displaystyle f_{odd}(\theta_0+\pi)=\min_{\theta}f_{odd}(\theta)=-f_{odd}(\theta_0)$.

Because $f_{even}$ is $\pi$-periodic, then $f_{even}(\theta_0)=f_{even}(\theta_0+\pi)$. One can see that:
\begin{itemize}
\item If $f_{even}(\theta_0)\geqslant 0$, then 

$\displaystyle\max_{\theta}|f(\theta)|\geqslant |f(\theta_0)|=f(\theta_0)\geqslant f_{odd}(\theta_0)=\max_{\theta}|f_{odd}(\theta)|$.
\item If $f_{even}(\theta_0)<0$, then

$\displaystyle\max_{\theta}|f(\theta)|\geqslant |f(\theta_0+\pi)|=-f(\theta_0+\pi)\geqslant -f_{odd}(\theta_0+\pi)=\max_{\theta}|f_{odd}(\theta)|$.
\end{itemize}
\end{proof}

\begin{defn}
Let $p_{\M}$ be the Minkowski support function of a positively oriented oval $\M$ of length $L_{\M}$. Then 
\begin{align}\label{SupportWM}
p_{W_{\M}}(\theta)=\frac{L_{\M}}{2\pi}+\frac{p_{\M}(\theta)-p_{\M}(\theta+\pi)}{2}
\end{align}
will be the support function of a curve $W_{\M}$ which will be called the \textit{Wigner caustic type curve associated with $\M$}.
\end{defn}

\begin{prop}
Let $W_{\M}$ be the Wigner caustic type curve associated with an oval $\M$. Then $W_{\M}$ has following properties:
\begin{enumerate}[(i)]
\item $W_{\M}$ is an oval of constant width,
\item $L_{W_{\M}}=L_{\M}$,
\item $\Eq_{\frac{1}{2}}(W_{\M})=\Eq_{\frac{1}{2}}(\M)$,
\item $A_{W_M}\geqslant A_{\M}$ and the equality holds if and only if $\M$ is a curve of constant width,
\item ${W_M}=\M$ if and only if $\M$ is a curve of constant width.
\end{enumerate}
\end{prop}
\begin{proof}
By (\ref{SupportWM}), to prove that $W_{\M}$ we will show $\rho_{W_{\M}}(\theta)>0$.

By (\ref{Fourierofp}) and (\ref{CurvatureM}) the radius of a curvature of $W_{M}$ is equal to 
\begin{align}
\rho_{W_M}(\theta) &=p_{W_M}(\theta)+p''_{W_M}(\theta)\\ 
\label{CurvatureWM}	&=a_0+\sum_{n=1, \text{n is odd}}^{\infty}(-n^2+1)(a_n\cos n\theta+b_n\sin n\theta)
\end{align}
Because $\displaystyle\rho_{M}(\theta)=a_0+\sum_{n=1}^{\infty}(-n^2+1)(a_n\cos n\theta+b_n\sin n\theta)>0$, then also by (\ref{CurvatureWM}) inequality $\rho_{W_M}>0$ holds. This is a consequence of that the range of \linebreak $\displaystyle\left|\sum_{\text{n is odd}}(-n^2+1)(a_n\cos n\theta+b_n\sin n\theta)\right|$ is a subset of the range of \linebreak $\displaystyle\left|\sum_n(-n^2+1)(a_n\cos n\theta+b_n\sin n\theta)\right|$ - see Lemma \ref{LemmaTrig}.

To check that $W_M$ is a curve of constant width let us notice that \linebreak $p_{W_M}(\theta)+p_{W_M}(\theta+\pi)$ is a constant.

By (\ref{CauchyFormula}) one can get
\begin{align*}
L_{W_M} &=\int_0^{2\pi}p_{W_M}(\theta)d\theta=\int_0^{2\pi}\left(\frac{L_M}{2\pi}+\frac{p_{M}(\theta)-p_{M}(\theta+\pi)}{2}\right)d\theta=L_M.
\end{align*}

By (\ref{ParameterizationEqM}) the support function of the Wigner caustic of $\M$ is \linebreak $\displaystyle\frac{p(\theta)-p(\theta+\pi)}{2}$. One can check that also this is the support function of the Wigner caustic of $W_M$.
Then using the improved isoperimetric inequality it is easy to show that $A_{W_M}\geqslant A_M$ and the equality holds if and only if $M$ is a curve of constant width.
\end{proof}

\begin{thm}\label{ThmStabIneqMax}
Let $K$ be strictly convex domain of area $A_{\partial K}$ and perimeter $L_{\partial K}$ and let $\widetilde{A}_{E_{\frac{1}{2}}(\partial K)}$ denote the oriented area of the Wigner caustic of $\partial K$. Let $W_{K}$ denote the convex body for which $\partial W_{K}$ is the Wigner caustic type curve associated with $\partial K$. Then
\begin{align}\label{StabIneqMax}
L_{\partial K}^2-4\pi A_{\partial K}-8\pi\left|\widetilde{A}_{\Eq_{\frac{1}{2}}(\partial K)}\right|\geqslant 4\pi^2 d_{\infty}^2(K, W_K),
\end{align}
where equality holds if and only if $\partial K$ is a curve of constant width.
\end{thm}

\begin{proof}

Because of (\ref{Fourierofp}), (\ref{Fourierofpprime}), (\ref{CauchyFormula}), (\ref{BlaschkeFormula}), the support functions $p_{\partial K}$ and $p_{\partial W_K}$ have the following Fourier series:

\begin{align}\label{FourierOfMWM}
p_{\partial K}(\theta) &=a_0+\sum_{n=1}^{\infty}(a_n\cos n\theta+b_n\sin n\theta),\\
\nonumber p_{\partial W_K}(\theta) &=a_0+\sum_{n=1, \text{n is odd}}^{\infty}(a_n\cos n\theta+b_n\sin n\theta).
\end{align}

Also one can get the Fourier series of $\Phi$ (see (\ref{IneqConvexProblem})):

\begin{align}\label{FourierOfPhi}
\Phi(K) &=L_{\partial K}^2-4\pi A_{\partial K}-8\pi\left|\widetilde{A}_{E_{\frac{1}{2}}(\M)}\right|\\
\nonumber	&=2\pi^2\sum_{n=2, \textit{n is even}}^{\infty}(n^2-1)(a_n^2+b_n^2).
\end{align}

One can check that $|a_n\cos n\theta+b_n\sin n\theta|\leqslant\sqrt{a_n^2+b_n^2}$ and then by (\ref{HausdorffDistance}) and H\"older's inequality:

\begin{align*}
d_{\infty}(K, W_K) &= \max_{\theta}\Big|p_{\partial K}(\theta)-p_{\partial W_K}(\theta)\Big|    \\
	&=\max_{\theta}\left|\sum_{n=2, \textit{n is even}}^{\infty}(a_n\cos n\theta+b_n\sin n\theta)\right|\\
	&\leqslant\max_{\theta}\left(\sum_{n=2, \textit{n is even}}^{\infty}|a_n\cos n\theta+b_n\sin n\theta|\right)\\
	&\leqslant\sum_{n=2, \textit{n is even}}\frac{1}{\sqrt{n^2-1}}\cdot \sqrt{n^2-1}\sqrt{a_n^2+b_n^2}\\
	&\leqslant\sqrt{\sum_{n=2, \textit{n is even}}\frac{1}{n^2-1}}\cdot \sqrt{\sum_{n=2, \textit{n is even}}(n^2-1)(a_n^2+b_n^2)}\\
	&=\sqrt{\frac{1}{2}}\cdot\sqrt{\frac{\Phi(K)}{2\pi^2}}.
\end{align*}
And the equality holds if and only if $a_{2m}=b_{2m}=0$ for all $m\in\mathbb{N}$, so $\partial K$ is a curve of constant width.
\end{proof}

\begin{thm}\label{ThmStabIneqL2}
Under the same assumptions of Theorem \ref{ThmStabIneqMax}, one gets
\begin{align}\label{StabIneqL2}
L_{\partial K}^2-4\pi A_{\partial K}-8\pi\left|\widetilde{A}_{\Eq_{\frac{1}{2}}(\partial K)}\right|\geqslant 6\pi d_2^2(K, W_K),
\end{align}
where equality holds if and only if $\partial K$ is a curve of constant width, or the Minkowski support function of $\partial K$ is in the form 
\begin{align*} p_{\partial K}(\theta)=a_0+a_2\cos 2\theta+b_2\sin 2\theta+\sum_{n=1, \text{n is odd}}^{\infty}(a_{n}\cos n\theta+b_{n}\sin n\theta).
\end{align*}
\end{thm}

\begin{proof}
Because of (\ref{FourierOfMWM}) and (\ref{FourierOfPhi})
\begin{align*}
d_2^2(K,W_K) &=\int_0^{2\pi}\Big|p_{\partial K}(\theta)-p_{\partial W_K}(\theta)\Big|^2d\theta\\
	&=\int_0^{2\pi}\left|\sum_{n=2, \textit{n is even}}(a_n\cos n\theta+b_n\sin\theta)\right|^2d\theta\\
	&=\pi\sum_{n=2, \textit{n is even}}(a_n^2+b_n^2)\\
	&\leqslant\frac{1}{6\pi}\cdot 2\pi^2\sum_{n=2, \textit{n is even}}(n^2-1)(a_n^2+b_n^2)\\
	&=\frac{1}{6\pi}\Phi(K).
\end{align*}
And the equality holds if and only if $a_{2m}=b_{2m}=0$ for all $m\in\mathbb{N}$, so $\partial K$ is a curve of constant width, or $\displaystyle p_{\partial K}(\theta)=a_0+a_2\cos 2\theta+b_2\sin 2\theta+\sum_{n=1, \text{n is odd}}^{\infty}(a_{n}\cos n\theta+b_{n}\sin n\theta)$.
\end{proof}

Let us consider the convex body $K$ for which (\ref{ExampleSupportFun}) is its Minkowski support function -- see Fig. \ref{FigExampleStabIneq}. We will check how close the right hand sides in the stability inequalities  (\ref{StabIneqMax}) and  (\ref{StabIneqL2}) are to be optimal.

\begin{align}\label{ExampleSupportFun}
p_{\partial K}(\theta)=10+2\cos 2\theta-\frac{1}{3}\sin 3\theta-\frac{1}{4}\cos 4\theta.
\end{align}

\begin{figure}[h]
\centering
\includegraphics[scale=0.4]{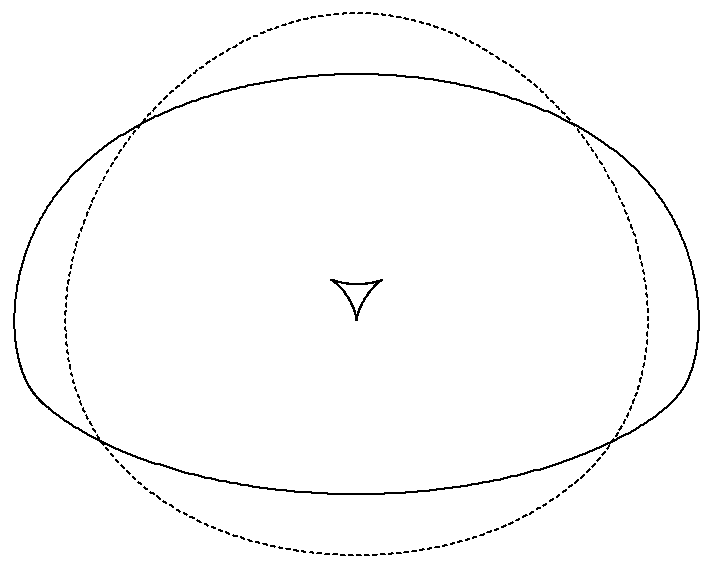}
\caption{A convex body $K$ for which $p_{\partial K}(\theta)=10+2\cos 2\theta-\frac{1}{3}\sin 3\theta-\frac{1}{4}\cos 4\theta$ is its Minkowski support function, the Wigner caustic type curve associated with $\partial K$ (the dashed one) and the Wigner caustic of $\partial K$.}
\label{FigExampleStabIneq}
\end{figure}
Then
\begin{align}
p_{\partial W_{K}}(\theta) & =10-\frac{1}{3}\sin 3\theta,\\
p_{\Eq_{\frac{1}{2}}(\partial K)}(\theta) &=-\frac{1}{3}\sin 3\theta.
\end{align}
One can check that
\begin{align}
L_{\partial K} &= 20\pi,\\
A_{\partial K} &=\frac{26809}{288}\pi, \\
\widetilde{A}_{\Eq_{\frac{1}{2}}(\partial K)} &=-\frac{2\pi}{9},\\
\label{lhs} L_{\partial K}^2-4\pi A_{\partial K}-8\pi\left|\widetilde{A}_{\Eq_{\frac{1}{2}}(\partial K)}\right| &=25.875\pi^2,
\end{align}
\begin{align}
\label{rhs1} 4\pi^2d_{\infty}^2(K, W_K) =4\pi^2\left(\max_{\theta}\left|2\cos 2\theta-\frac{1}{4}\cos 4\theta\right|\right)^2 &=20.25\pi^2,\\
\label{rhs2} 6\pi d_2^2(K,W_K) =6\pi\int_0^{2\pi}\left(2\cos 2\theta-\frac{1}{4}\cos 4\theta\right)^2d\theta &=24.375\pi^2.
\end{align}

Then by (\ref{lhs}), (\ref{rhs1}) the stability inequality (\ref{StabIneqMax}) in Theorem \ref{ThmStabIneqMax} is in the following form: 
\begin{align}
25.875\pi^2\geqslant 20.25\pi^2,
\end{align}
and by (\ref{lhs}), (\ref{rhs2}) the stability inequality (\ref{StabIneqL2}) in Theorem \ref{ThmStabIneqL2} is in the following form: 
\begin{align}
25.875\pi^2\geqslant 24.375\pi^2.
\end{align}

\section*{Acknowledgements}
The author would like to thank Professor Wojciech Domitrz for the helpful discussions.

\bibliographystyle{amsalpha}

\begin{thebibliography}{AAAA}

\bibitem [1] {B1} Berry, M.V.: \emph{Semi-classical mechanics in phase space: a study of Wigner's function}, Philos. Trans. R. Soc. Lond. A 287, 237-271 (1977).

\bibitem [2] {B2} Barbier, E. (1860): \emph{Note sur le probleme de l'aiguille et le jeu du joint couvert} (PDF), Journal de math\'ematiques pures et appliqu\'ees, 2e s\'erie (in French) 5: 273--286. See in particular pp. 283--285.

\bibitem [3] {C1} Chavel, I.: \emph{Isoperimetric Inequalities}. Differential Geometric and Analytic Perspectives. Cambridge University Press, 2001.

\bibitem [4] {C2} Craizer, M.: \emph{Iterates of Involutes of Constant Width Curves in the Mikowski Plane}, Beitr\"age zur Algebra und Geometrie, 55(2), 479--496, 2014.

\bibitem [5] {CDR1} Craizer, M., Domitrz, W., Rios, P. de M.: \emph{Even Dimensional Improper Affine Spheres}, J. Math. Anal. Appl. 421 (2015), pp. 1803--1826.

\bibitem [6] {DMR1} Domitrz, W., Manoel, M., Rios, P. de M.: \emph{The Wigner caustic on shell and singularities of odd functions} , Journal of Geometry and Physics 71(2013), pp. 58-72

\bibitem [7] {DR1} Domitrz, W., Rios, P. de M.: \emph{ Singularities of equidistants and Global Centre Symmetry sets of Lagrangian submanifolds}, Geom. Dedicata 169 (2014), pp. 361-382.

\bibitem [8] {DRS1} Domitrz, W., Rios, P. de M., Ruas, M. A. S.: \emph{Singularities of affine equidistants: projections and contacts}, J. Singul. 10 (2014), 67-81

\bibitem [9] {DZ1} Domitrz, W., Zwierzy\'nski, M.: \emph{The geometry of the Wigner caustic and affine equidistants of planar curves}, preprint.

\bibitem [10] {F1} Fisher, J. C.: \emph{Curves of Constant Width from a Linear Viewpoint}, Mathematics Magazine Vol. 60, No. 3 (Jun., 1987), pp. 131-140

\bibitem [11] {F2} Fuglede, B.: \emph{Stability in the isoperimetric problem}, Bull. London Math. Soc. 18 (1986), 599 -- 605

\bibitem [12] {G1} Gao, X.: \emph{A Note on the Reverse Isoperimetric Inequality}. Results. Math. 59 (2011), 83--90

\bibitem [13] {G3} Giblin, P.J.: \emph{Affinely invariant symmetry sets}. Geometry and Topology of Caustics, Banach Center Publications, vol.82 (2008), p.71-84.

\bibitem [14] {GWZ1} Giblin, P.J., Warder, J.P., Zakalyukin, V.M.: \emph{ Bifurcations of affine equidistants}, Proceedings of the Steklov Institute of Mathematics 267 (2009), 57--75.

\bibitem [15] {GZ1} Giblin P.J, Zakalyukin, V.M: \emph{Singularities of Centre Symmetry Sets}. Proc.
London Math. Soc. (3) 90 (2005), 132--166.

\bibitem [16] {G4} Groemer, H.: \emph{Geometric applications of Fourier series and spherical harmonics}, Encyclopedia of Mathematics and its Applications, vol. 61. Cambridge University
Press, Cambridge (1996)

\bibitem [17] {G2} Groemer, H.: \emph{Stability properties of geometric inequalities}, in: P.M. Gruber, J.M. Wills (Eds.), \emph{Handbook of Convex Geometry}, North--Holland, 1993, pp. 125--150.

\bibitem [18] {H2} Hsiung, C. C.: \emph{A First Course in Differential Geometry. Pure Appl. Math.}, Wiley, New York 1981. 

\bibitem [19] {H3} Hurwitz, A.: \emph{Sur quelque applications geometriques des series Fourier}, Ann. Sci. Ecole Normal Sup. (3) 19 (1902). 357--408.

\bibitem [20] {L1} Lawlor, G.: \emph{A new area-maximization proof for the circle}. Math. Intell. 20 (1998), 29--31

\bibitem [21] {PX1} Pan, S.L.; Xu, H.: \emph{Stability of a reverse isoperimetric inequality}. J. Math. Anal. Appl. 350 (2009), no. 1, 348-353.

\bibitem [22] {RZ1} Reeve, G. M., Zakalyukin, V. M.: \emph{Singularities of the Minkowski set and affine equidistants for a curve and a surface}. Topology Appl. 159 (2012), no. 2, 555--561.

\bibitem [23] {R1} Schender, R.: \emph{Convex Bodies: The Brunn-Minkowski Theory}, Cambridge, 2013.

\bibitem [24] {S2} Schneider, R.: \emph{A stability estimate for the Aleksandrov--Fenchel inequality, wich an application to mean curvature}, Manuscripta Math. 69 (1990) 291--300.

\bibitem [25] {S1} Steiner, J.: \emph{Sur le maximum et le minimum des figures dans le plan, sur la sph\'ere, et dans l'espace en g\'en\'eral}, I and II. J. Reine Angew. Math. (Crelle) 24 (1842), 93-152 and 189-250.

\bibitem [26] {Z1} Zakalyukin, V.M.: \emph{Envelopes of families of wave fronts and control theory}, Proc. Steklov
Math. Inst. 209 (1995), 133-142.











\end{thebibliography}

\end{document}